\documentclass[a4paper,12pt,reqno]{amsart}
\usepackage{amssymb}
\usepackage{latexsym}
\usepackage{amsfonts}
\usepackage{amsmath}
\usepackage{amsthm}
\usepackage{multicol}
\theoremstyle{definition}
\newtheorem{definition}{Definition}[section]
\newtheorem{proposition}[definition]{Proposition}
\newtheorem{theorem}[definition]{Theorem}

\newtheorem{remark}[definition]{Remark}

\def\div{\mbox{div}\,}

\def\<{\mathop{<}}
\def\>{\mathop{>}}

\newcommand{\spt}{\mathrm{spt}\,}

\newcommand{\dist}{\mathrm{dist}\,}
\numberwithin{equation}{section}
\def\Xint#1{\mathchoice
{\XXint\displaystyle\textstyle{#1}}%
{\XXint\textstyle\scriptstyle{#1}}%
{\XXint\scriptstyle\scriptscriptstyle{#1}}%
{\XXint\scriptscriptstyle\scriptscriptstyle{#1}}%
\!\int}
\def\XXint#1#2#3{{\setbox0=\hbox{$#1{#2#3}{\int}$}
\vcenter{\hbox{$#2#3$}}\kern-.5\wd0}}

\def\dashint{\Xint-}

\title[Existence of volume preserving mean curvature flow]{Existence of weak solution for volume preserving mean curvature flow via phase field method}
\author[K. Takasao]{Keisuke Takasao \\ Graduate School of Mathematical Sciences
University of Tokyo \\ Komaba 3-8-1, Meguro
JP-153-8914 Tokyo
Japan}
\email{takasao@ms.u-tokyo.ac.jp}
\keywords{mean curvature flow, Allen-Cahn equation, phase field method}
\subjclass[2010]{Primary~35K93, Secondary~53C44}
\thanks{The author is grateful to Professor Yoshikazu Giga, Professor Yoshihiro Tonegawa, Professor Noriaki Yamazaki and Professor Tomoyuki Suzuki for numerous comments}
\date{}
\begin{document}
\maketitle
\begin{abstract}
We study the phase field method for the volume preserving mean curvature flow. Given an initial $C^1$ hypersurface we proved the existence of the weak solution for the volume preserving mean curvature flow via the reaction diffusion equation with a nonlocal term. We also show the monotonicity formula and the density upper bound for the reaction diffusion equation.
\end{abstract}
\section{Introduction}
Let $U_t \subset \mathbb{R}^d$ be a bounded open set and have a smooth boundary $M_t$ for $t\in [0,T)$. The family of hypersurfaces $\{ M _t \}_{t\in [0,T)}$ is called the volume preserving mean curvature flow if the velocity vector $v$ of $M _t$ is given by
\begin{equation}
v= h-\langle h\cdot \nu \rangle \nu \quad \text{on} \ M_t,
\label{mcf}
\end{equation}
where $h$ and $\nu$ are the mean curvature vector and the inner unit normal vector of $M_t$ respectively, and $\langle h \cdot \nu \rangle$ is given by
\[ \langle h \cdot \nu \rangle := \frac{1}{\mathcal{H}^{d-1} (M_t)}\int _{M_t} h \cdot \nu \, d\mathcal{H}^{d-1}.\]
Here $\mathcal{H}^{d-1}$ is the $(d-1)$-dimensional Hausdorff measure. By \eqref{mcf}, this flow has the volume preserving property, that is
\begin{equation}
\frac{d}{dt} \mathcal{L}^{d} (U_t) =-\int _{M_t} v\cdot \nu \, d\mathcal{H}^{d-1}=0,
\label{vp}
\end{equation}
where $\mathcal{L}^d$ is the $d$-dimensional Lebesgue measure. By \eqref{vp}, we obtain
\begin{equation}
\begin{split}
&\frac{d}{dt}\mathcal{H}^{d-1} (M_t) =- \int _{M_t} h\cdot v \, d\mathcal{H}^{d-1}= -\int _{M_t} (v+\langle h\cdot \nu \rangle \nu) \cdot v \, d\mathcal{H}^{d-1}\\
=&  -\int _{M_t} |v|^2 \, d\mathcal{H}^{d-1} -\langle h \cdot \nu \rangle \int _{M_t} \nu \cdot v \, d\mathcal{H}^{d-1} =-\int _{M_t} |v|^2 \, d\mathcal{H}^{d-1} .
\end{split}
\label{v2}
\end{equation}

The time global existence of the classical solution to \eqref{mcf} for convex initial data is proved by Gage \cite{gage} ($d=2$) and Huisken \cite{huisken1987} ($d \geq 2$). Escher and Simonett \cite{escher-simonett} proved the short time existence of the solution to \eqref{mcf} for smooth initial data, and they show that if $M_0$ is sufficiently close to a Euclidean sphere, then there exists a time global solution.
Li \cite{li} also proved that if the traceless second fundamental form of initial data is sufficiently small then there exists a time global solution. Recently, Mugnai, Seis and Spadaro \cite{mugnai-seis-spadaro} proved the existence of the global distributional solution for \eqref{mcf} by using a variational approach.

Next we mention the approximation of the volume preserving mean curvature flow via the phase field method. Let $\varepsilon\in (0,1)$ and $\Omega$ be the torus, that is $\Omega:=\mathbb{T}^d=(\mathbb{R}/\mathbb{Z})^{d}$. We also use $\Omega$ to a set $[0,1)^d \subset \mathbb{R}^d$. Rubinstein and Sternberg~\cite{RubinsteinSternberg} considered the following Allen-Cahn equation with a nonlocal term:
\begin{equation}
\left\{ 
\begin{array}{ll}
\varepsilon \varphi ^{\varepsilon} _t =\varepsilon \Delta \varphi ^{\varepsilon} -\dfrac{W '(\varphi ^{\varepsilon})}{\varepsilon }+ \lambda_{1} ^\varepsilon ,& (x,t)\in \Omega \times (0,\infty),  \\
\varphi ^{\varepsilon} (x,0) = \varphi _0 ^{\varepsilon} (x) ,  &x\in \Omega,
\end{array} \right.
\label{rs}
\end{equation}
where $\displaystyle W(s):=\frac{(1-s^2)^2}{2}$ and $\displaystyle\lambda_{1}(t) :=\dashint _{\Omega }\frac{W' (\varphi ^{\varepsilon})}{\varepsilon} \, dx =\frac{1}{\mathcal{L}^d(\Omega)} \int _{\Omega }\frac{W' (\varphi ^{\varepsilon})}{\varepsilon} \, dx$.

\eqref{rs} has the volume preserving property, that is
\begin{equation}
\frac{d}{dt} \int _{\Omega } \varphi ^{\varepsilon} \, dx= \int _{\Omega } \varphi _t ^{\varepsilon} \, dx=0.
\label{phi}
\end{equation}
By \eqref{phi} we obtain
\begin{equation}
\begin{split}
&\frac{d}{dt}\int _{\Omega}  \frac{\varepsilon |\nabla \varphi^\varepsilon | ^2}{2} +\frac{W (\varphi^\varepsilon )}{\varepsilon} \, dx  
=\int  _{\Omega} \Big( \varepsilon \nabla \varphi^\varepsilon \cdot \nabla \varphi _t ^\varepsilon +\frac{W' (\varphi^\varepsilon )}{\varepsilon}\varphi _t ^\varepsilon \Big) \, dx \\
=&  \int_{\Omega }  \Big( -\varepsilon \Delta \varphi^\varepsilon +\frac{W' (\varphi ^\varepsilon )}{\varepsilon} \Big)\varphi^\varepsilon _t \, dx=  \int_{\Omega }  (- \varepsilon \varphi _t ^\varepsilon +\lambda _{1}^\varepsilon ) \varphi^\varepsilon _t \, dx \\
=&  - \int_{\Omega }  \varepsilon( \varphi _t ^\varepsilon)^2 \, dx +\lambda _{1}^\varepsilon \int _{\Omega } \varphi ^{\varepsilon} _t \, dx =- \int_{\Omega }  \varepsilon( \varphi _t ^\varepsilon)^2 \, dx.
\end{split}
\label{deri}
\end{equation}
By using the following approximate expressions (see \cite{Ilmanen})
\begin{equation}
\mathcal{H}^{d-1} (M_t) \approx \frac{1}{\sigma}\int _{\Omega}  \frac{\varepsilon |\nabla \varphi^\varepsilon | ^2}{2} +\frac{W (\varphi^\varepsilon )}{\varepsilon} \, dx   \quad \text{and} \quad  \int _{\Omega} |v|^2 \, d\mathcal{H}^{d-1} \approx \frac{1}{\sigma} \int _{\Omega} \varepsilon (\varphi _t ^\varepsilon)^2 \, dx,
\label{approx1.7}
\end{equation}
\eqref{deri} corresponds to \eqref{v2}. Here $\sigma := \int_{-1} ^{1} \sqrt{2W ( s )} \, ds$. Bronsard and Stoth \cite{bronsard-stoth} studied the singular limit of radially symmetric solutions of \eqref{rs}. Chen, Hilhorst and Logak \cite{chen-hilhorst-logak} proved that the zero level set of the solution of \eqref{rs} converges to the classical solution of the volume preserving mean curvature flow under the suitable conditions. 

Recently, Brassel and Bretin~\cite{brassel-bretin} studied following equation:
\begin{equation}
\left\{ 
\begin{array}{ll}
\varepsilon \varphi ^{\varepsilon}_t =\varepsilon \Delta \varphi ^{\varepsilon} -\dfrac{W' (\varphi ^{\varepsilon})}{\varepsilon }+\lambda_{2} ^\varepsilon \sqrt{2W (\varphi ^{\varepsilon})}  ,& (x,t)\in \Omega \times (0,\infty),  \\
\varphi ^{\varepsilon} (x,0) = \varphi _0 ^{\varepsilon} (x) ,  &x\in \Omega,
\end{array} \right.
\label{bb}
\end{equation}
where 
\[ \lambda_{2}=\lambda_{2}(t) := \frac { \int _{\Omega }W' (\varphi ^{\varepsilon})/\varepsilon \, dx }{ \int _{\Omega } \sqrt{2W (\varphi ^{\varepsilon})}  \, dx } .\]
The solution of \eqref{bb} has also the property \eqref{phi}. Alfaro and Alifrangis \cite{alfaro-alifrangis} showed the convergence of \eqref{bb} to the classical solution for \eqref{mcf}. \cite{brassel-bretin} showed that the numerical experiments via \eqref{bb} is better than \eqref{rs}. But \eqref{bb} has not the properties such as \eqref{deri}.

Whether the solution for \eqref{rs} or \eqref{bb} converges to the time global weak solution of the volume preserving mean curvature flow or not is an open problem, due to the difficulty of estimates of the Lagrange multipliers (see Remark \ref{difficulty}).

In this paper, we consider the following reaction diffusion equation studied by Golovaty~\cite{golovaty}:
\begin{equation}
\left\{ 
\begin{array}{ll}
\varepsilon \varphi ^{\varepsilon} _t =\varepsilon \Delta \varphi ^{\varepsilon} -\dfrac{W' (\varphi ^{\varepsilon})}{\varepsilon }+\lambda ^\varepsilon \sqrt{2W(\varphi ^\varepsilon)} ,& (x,t)\in \Omega \times (0,\infty),  \\
\varphi ^{\varepsilon} (x,0) = \varphi _0 ^{\varepsilon} (x) ,  &x\in \Omega,
\end{array} \right.
\label{ac}
\end{equation}
where
\begin{equation*}
\begin{split}
\lambda^\varepsilon =\lambda^\varepsilon (t) :=\frac{- \int_{\Omega} \sqrt{2W(\varphi ^\varepsilon)} (\varepsilon \Delta \varphi ^\varepsilon - W'(\varphi ^\varepsilon ) / \varepsilon  ) \, dx }{ 2\int_{\Omega}  W(\varphi^\varepsilon) \, dx }.
\end{split}
\end{equation*}
Note that by the integration by parts, we have
\[ \lambda^\varepsilon =\frac{-2 \int_{\Omega} \varphi^\varepsilon ( \varepsilon |\nabla \varphi^\varepsilon | ^2/2 +W(\varphi ^\varepsilon )/\varepsilon ) \, dx }{ \int_{\Omega} W(\varphi^\varepsilon) \, dx }. \]
\cite{golovaty} studied the asymptotic behavior of the radially symmetric solutions for \eqref{ac}. Define $k(s):=\int_{0}^{s} \sqrt{2W(\tau)} \,d\tau =s-\frac{1}{3} s^3$. \eqref{ac} has a property similar to \eqref{phi}, that is,
\begin{equation}
\frac{d}{dt}\int_{\Omega} k(\varphi ^\varepsilon) \, dx=\int_{\Omega}\varphi ^\varepsilon _t  \sqrt{2W(\varphi ^\varepsilon)}  \, dx =0.
\label{h}
\end{equation}
We compute that
\begin{equation}
\begin{split}
&\frac{d}{dt}\int _{\Omega}  \frac{\varepsilon |\nabla \varphi^\varepsilon | ^2}{2} +\frac{W (\varphi^\varepsilon )}{\varepsilon} \, dx  
=\int  _{\Omega} \Big( \varepsilon \nabla \varphi^\varepsilon \cdot \nabla \varphi _t ^\varepsilon +\frac{W' (\varphi^\varepsilon )}{\varepsilon}\varphi _t ^\varepsilon \Big) \, dx \\
=&  \int_{\Omega }  \Big( - \varepsilon \Delta \varphi^\varepsilon +\frac{W' (\varphi ^\varepsilon )}{\varepsilon } \Big)\varphi^\varepsilon _t \, dx=  \int_{\Omega }  ( - \varepsilon \varphi _t ^\varepsilon +\lambda^\varepsilon \sqrt{2W(\varphi ^\varepsilon)} ) \varphi^\varepsilon _t \, dx \\
=&  - \int_{\Omega }  \varepsilon( \varphi _t ^\varepsilon)^2 \, dx + \lambda ^\varepsilon \int _{\Omega } \varphi ^{\varepsilon} _t \sqrt{2W(\varphi ^\varepsilon )} \, dx =- \int_{\Omega }  \varepsilon( \varphi _t ^\varepsilon)^2 \, dx,
\end{split}
\label{deri2}
\end{equation}
where \eqref{h} is used. By using \eqref{approx1.7}, \eqref{deri2} also corresponds to \eqref{v2}. 

The main result in this paper is the time global existence of the weak solution for \eqref{mcf} by using \eqref{ac} (see Theorem \ref{mainresults}). The weak solution is called $L^2$-flow defined by Mugnai and R{\"o}ger~\cite{mugnai-roger2008} and the definition is similar to Brakke's mean curvature flow \cite{brakke}. In Proposition \ref{keyprop} and Proposition \ref{proph}, we obtain the $L^2$ estimates for $\lambda ^\varepsilon$ and the generalized mean curvature. Those are the key estimates for the existence theorem.

The monotonicity formula for the mean curvature flow is proved by Huisken~\cite{huisken1990}. Ilmanen proved the $\varepsilon$-version of Huisken's monotonicity formula of the Allen-Cahn equation for the mean curvature flow~\cite{Ilmanen}. In this paper we show the monotonicity formula for \eqref{ac} (Proposition \ref{propmono}). We also obtain the upper density bounds for \eqref{ac} by using the monotonicity formula (Proposition \ref{propdens}).
 
The organization of the paper is as follows. In Section 2 we set out the basic definitions and explain the main results. In Section 3 we show some energy estimates for \eqref{ac} and the $L^2$ estimates of the Lagrange multiplier $\lambda^\varepsilon$. We also prove the monotonicity formula and density upper bounds for \eqref{ac}. In
Section 4 we prove the main results.
\section{Preliminaries and main results}
We recall some notations from geometric measure theory and refer to \cite{allard,brakke,evansgariepy,federer,simon} for more details. For $r>0$ and $a\in \mathbb{R}^d$ we define $B_r (a):=\{ x\in \mathbb{R}^d \, | \, |x-a|<r \}$. Set $\omega _d :=\mathcal{L}^d (B_1 (0))$. We denote the space of bounded variation functions on $\mathbb{R}^d$ as $BV (\mathbb{R}^d)$. We write the characteristic function of a set $A \subset \mathbb{R}^d$ as $\chi _{A}$. For a set $A\subset \mathbb{R}^d$ with finite perimeter, we denote the total variation measure of the distributional derivative $\nabla \chi _{A}$ by $\| \nabla \chi _{A} \|$. For measures $\mu$ and $\tilde\mu$ such that $\tilde\mu$ is absolutely continuous with respect to $\mu$, we write the Radon-Nikodym derivative by $\frac{d\tilde\mu}{d\mu}$. For $a=(a_1 ,a_2,\dots ,a_d)$, $b=(b_1 ,b_2,\dots ,b_d) \in \mathbb{R}^d$ we denote $a\otimes b :=(a_i b_j)$. For $A=(a_{ij}),B=(b_{ij}) \in \mathbb{R}^{d\times d}$, we define
\begin{equation*}
A:B:=\sum_{i,j=1}^d a_{ij}b_{ij}.
\end{equation*}
Let $G_k (\mathbb{R}^d)$ be the Grassman manifold of unoriented $k$-dimensional subspaces in $\mathbb{R}^d$. Let $S\in G_k (\mathbb{R}^d)$. We also use $S$ to denote the $d$ by $d$ matrix representing the orthogonal projection $\mathbb{R}^d\to S$. Especially, if $k=d-1$ then the projection for $S\in G_{d-1}(\mathbb{R}^d)$ is given by $S=I-\nu\otimes \nu$, where $I$ is the identity matrix and $\nu$ is the unit normal vector of $S$.

We call a Radon measure on $\mathbb{R}^d\times G_k (\mathbb{R}^d)$ a general $k$-varifold in $\mathbb{R}^d$. We denote the set of all general $k$-varifolds by $\mathbf{V}_k(\mathbb{R}^d)$. Let $V \in \mathbf{V}_k (\mathbb{R}^d)$. We define a mass measure of $V$ by
\[ \| V \| (A) := V( (\mathbb{R} ^d \cap A ) \times G_k (\mathbb{R}^d) )  \]
for any Borel set $A \subset \mathbb{R}^d$.
We also denote
\[ \| V \| (\phi ) :=  \int _{ \mathbb{R} ^d  \times G_k (\mathbb{R}^d) } \phi (x) \, dV(x,S) \quad \text{for} \quad \phi \in C_c (\mathbb{R}^d). \]
The first variation $\delta V:C_c ^1 (\mathbb{R}^d ;\mathbb{R}^d)\to \mathbb{R}$ of $V\in \mathbf{V}_k(\mathbb{R}^d)$ is defined by
\[ \delta V(g):= \int _{\mathbb{R} ^d \times G_k (\mathbb{R} ^d)} \nabla g(x) : S \, dV(x,S) \quad \text{for} \quad g \in C_c ^1 (\mathbb{R}^d ;\mathbb{R}^d).\]
We define a total variation $\| \delta V \|$ to be the largest Borel regular measure on $\mathbb{R}^d$ determined by
\[ \|\delta V \| (G) := \sup \{ \delta V(g)  \, | \, g\in C_c ^1 (G;\mathbb{R}^d), \ |g|\leq 1 \} \]
for any open set $G\subset \mathbb{R}^d$. If $\|\delta V \|$ is locally bounded and absolutely continuous with respect to $\|V\|$, then by the Radon-Nikodym theorem, there exists a $\|V\|$-measurable function 
$ h(x)$ with values in $\mathbb{R}^d$ such that
\[  \delta V  (g) = - \int _{\mathbb{R} ^d } h(x)  \cdot g(x)\, d\|V\|(x)  \quad \text{for} \quad g\in C_c (\mathbb{R}^d;\mathbb{R}^d). \]
We call $h$ the generalized mean curvature vector of $V$. 

We call a Radon measure $\mu $ $k$-rectifiable if $\mu $ is represented by $\mu = \theta \mathcal{H} ^k \lfloor M$, that is, $\mu (\phi):=\int _{\mathbb{R}^d} \phi \, d\mu = \int _M \phi \theta \, d\mathcal{H}^k$ for any $\phi \in C_c (\mathbb{R}^d)$. Here $M$ is countably $k$-rectifiable and $\mathcal{H}^k$-measurable, and $\theta \in L^1 _{loc} (\mathcal{H}^k \lfloor M)$ is positive valued $\mathcal{H}^k$-a.e. on $M$. Moreover if  $\theta$ is positive and integer-valued $\mathcal{H}^k$-a.e. on $M$ then we call $\mu$ $k$-integral. 
For a $k$-rectifiable Radon measure $\mu= \theta \mathcal{H} ^k \lfloor M$ we define a unique $k$-varifold $V$ by
\begin{equation*}
\int _{\mathbb{R} ^d \times G_k (\mathbb{R} ^d)} \phi (x,S) \, dV(x,S)
:= \int _{\mathbb{R}^d} \phi (x,T_x M ) \theta(x) \, d\mathcal{H}^k (x) \qquad \text{for} \ \phi \in C_c(\mathbb{R}^d \times G_k (\mathbb{R}^d)), 
\end{equation*}
where $T_x M$ is the approximate tangent space of $M$ at $x$. Note that $T_x M$ exists $\mathcal{H}^k$-a.e. on $M$ in this assumption, and $\mu =\|V\|$ under this correspondence. 
The following definition is similar to the formulation of Brakke's mean curvature flow~\cite{brakke}:
\begin{definition}[$L^2$-flow~\cite{mugnai-roger2008}]
Let $T>0$ and $\{ \mu _t \}_{t\in (0,T)}$ be a family of Radon measures on $\mathbb{R}^d$. Set $d\mu :=d\mu_t dt$. We call $\{ \mu _t \}_{t\in (0,T)}$ $L^2$-flow if the following hold:
\begin{enumerate}
\item $\mu _t$ is $(d-1)$-rectifiable and has a generalized mean curvature vector $h \in L^2 (\mu _t;\mathbb{R}^d)$ a.e. $t\in (0,T)$, 
\item and there exist $C>0$ and a vector field $v\in L^2 (\mu ,\mathbb{R}^d)$ such that
\begin{equation}
v(x,t)\perp T_x \mu _t \quad \text{for} \ \mu \text{-a.e.} \ (x,t) \in \mathbb{R}^d\times (0,T)
\label{velo1}
\end{equation}
and
\begin{equation}
\Big| \int _0 ^T \int _{\mathbb{R}^d} (\eta _t + \nabla \eta \cdot v ) \, d\mu_t dt \Big| \leq C \| \eta \|_{C^0 (\mathbb{R}^d\times (0,T))}
\label{velo2}
\end{equation}
for any $\eta \in C_c ^1 (\mathbb{R}^d\times (0,T))$. Here $T_x \mu _t $ is the approximate tangent space of $\mu _t$ at $x$.
\end{enumerate}
Moreover $v\in L^2 (\mu ,\mathbb{R}^d)$ with \eqref{velo1} and \eqref{velo2} is called a generalized velocity vector.
\end{definition}
Set $q^\varepsilon (r) := \tanh (\frac{r}{\varepsilon})$ for $r\in \mathbb{R}$ and $\varepsilon >0$. Then the following hold:
\begin{enumerate}
\item $q^{\varepsilon}$ is a solution for
\begin{equation}
\frac{\varepsilon ( q^{\varepsilon} _r)^2 }{2} = \frac{W (q ^{\varepsilon})}{\varepsilon} \qquad \text{and} \qquad 
q^{\varepsilon} _{rr}= \frac{W' (q ^{\varepsilon})}{\varepsilon ^2}
\label{q}
\end{equation}
with $q^{\varepsilon}(0)=0, \ q^{\varepsilon}(\pm \infty)=\pm 1$ and $ q^{\varepsilon}_r (r)  >0$ for any $r\in \mathbb{R}$. 
\item By \eqref{q} we have
\begin{equation*}
\begin{split}
&\int_{\mathbb{R}}  \frac{\varepsilon ( q ^{\varepsilon}_r )^2}{2} + \frac{W (q^{\varepsilon} )}{\varepsilon} \,dr = 
\int_{\mathbb{R}} \sqrt{2W (q^{\varepsilon} )} q ^{\varepsilon}_r  \,dr \\
=& \int_{-1} ^{1} \sqrt{2W ( s )} \, ds=\sigma.
\end{split}
\end{equation*}
\end{enumerate}

Let $U_0 \subset \subset (0,1)^d$ be a bounded open set and we denote $M _0 :=\partial U_0$.
Throughout this paper, we assume the following:
\begin{enumerate}
\item There exists $D_0>0$ such that
\begin{equation}
\sup_{x\in \mathbb{R}^d, R>0}\frac{\mathcal{H}^{d-1} (M _0  \cap B_R (x)) }{\omega _{d-1}R^{d-1}} \leq D_0 \quad \text{(Density upper bounds)}.
\label{initialdata1}
\end{equation}
\item There exists a family of open sets $\{ U _0 ^i \}_{i=1} ^\infty$ such that $U _0 ^i$ have a  $C^3$ boundary $M _0 ^i$ such that $(U _0 ,M _0)$ be approximated strongly by $\{ (U _0 ^i ,M _0 ^i) \} _{i=1} ^\infty$, that is
\begin{equation}
\lim _{i\to \infty} \mathcal{L}^d (U _0 \triangle U_0 ^i) =0 \quad \text{and} \quad
\lim _{i \to \infty} \| \nabla \chi _{U ^i _0} \| = \|\nabla \chi _{U _0} \| \ \ \text{as measures.}
\label{initialdata2}
\end{equation}
\end{enumerate}
\begin{remark}
If $M _0$ is $C^1$, then \eqref{initialdata1} and \eqref{initialdata2} are satisfied.
\end{remark}
We extend $U_0 ^i$ and $M_0 ^i$ periodically to $\mathbb{R}^d$. Let $\{ \varepsilon _i \}_{i=1}^\infty$ be a sequence with $\varepsilon _i \downarrow 0$ as $i\to \infty$. For $U _0^i$ we define
\[
   r_{\varepsilon_i}(x)= \begin{cases}
    \dist(x,M_0 ^i), & x\in U _0^i ,\\
    -\dist(x,M_0 ^i), & x\notin U _0 ^i.
  \end{cases}
\]
We remark that $|\nabla r_{\varepsilon_i} |\leq 1$ a.e. $x\in \mathbb{R}^d$ and $r_{\varepsilon_i}$ is smooth near $M _0 ^i$. Let $\overline{r_{\varepsilon_i}}$ be a smoothing of $r_{\varepsilon_i}$ with $|\nabla\overline{ r_{\varepsilon_i}} |\leq 1$ and $|\nabla ^2 \overline{ r_{\varepsilon_i}} |\leq \varepsilon _i ^{-1}$ in $\mathbb{R}^d$, and $\overline{r_{\varepsilon_i}} =r_{\varepsilon_i}$ near $M _0 ^i$.

Define
\begin{equation}
\varphi ^{\varepsilon_i }_0(x):=q^{\varepsilon_i } (\overline{r_{\varepsilon_i}} (x)), \quad i\geq 1.
\label{initial}
\end{equation}
We remark that $\varphi ^{\varepsilon_i }_0$ is a periodic function with period $\Omega$. Then there exists a global solution $\varphi ^{\varepsilon_i }$ for \eqref{ac} with the initial data $\varphi ^{\varepsilon_i }_0$ (see \cite{golovaty}). We remark that $\int _\Omega k(\pm 1)\, dx= \pm \frac{2}{3}\mathcal{L}^d(\Omega) =\pm\frac{2}{3}$. So we may assume that there exists $\omega=\omega (U_0) >0$ such that
\begin{equation}
\Big|\int _\Omega k(\varphi ^{\varepsilon_i} _0)\, dx\Big| \leq \frac{2}{3} -\omega, \quad i\geq 1. 
\label{omega}
\end{equation}
Note that by \eqref{h} we have
\begin{equation}
\Big|\int _\Omega k(\varphi ^{\varepsilon_i} (x,t) )\, dx\Big|=\Big|\int _\Omega k(\varphi ^{\varepsilon_i} _0)\, dx\Big|\leq \frac{2}{3}  -\omega, \quad i\geq 1 , \ t\geq 0. 
\label{omega2}
\end{equation}
We remark that $W'(s)=\sqrt{2W(s)}=0$ if $s=\pm1$. Hence by the maximal principle we have
\begin{proposition}[\cite{golovaty}]
\begin{equation}
\sup_{(x,t)\in \Omega \times [0,\infty)} |\varphi ^{\varepsilon_i} (x,t)| \leq 1,\quad i\geq 1.
\label{sup1}
\end{equation}
\end{proposition}
We denote $\varphi ^{\varepsilon} := \varphi ^{\varepsilon_i}$ and extend $\varphi ^{\varepsilon}$ periodically to $\mathbb{R}^d$. We define a Radon measure $\mu _t ^{\varepsilon}$ by
\begin{equation*}
\mu _t ^{\varepsilon}(\phi) := \frac{1}{\sigma}\int _{\mathbb{R}^d} \phi \Big( \frac{\varepsilon |\nabla \varphi ^{\varepsilon}|^2}{2} + \frac{W (\varphi ^{\varepsilon} )}{\varepsilon} \Big) dx
\end{equation*}
for any $\phi \in C_c (\mathbb{R}^d)$. Moreover we define a Radon measure $\mu ^{\varepsilon}$ by
\begin{equation*}
\mu ^{\varepsilon}(\psi) := \frac{1}{\sigma}\int _{[0,\infty)} \int _{\mathbb{R}^d} \psi \Big( \frac{\varepsilon |\nabla \varphi ^{\varepsilon}|^2}{2} + \frac{W (\varphi ^{\varepsilon} )}{\varepsilon} \Big) dxdt
\end{equation*}
for any $\psi \in C_c (\mathbb{R}^d\times[0,\infty))$. By the definition of $\varphi _0 ^{\varepsilon} $ we obtain the following:
\begin{proposition}[Proposition 1.4 of \cite{Ilmanen}]\label{prop1} Let $\varphi _0 ^{\varepsilon} $ satisfy \eqref{initial}. Then the following hold:
\begin{enumerate}
\item There exists $D_1 =D_1 (D_0)>0$ such that for any $\varepsilon >0$, we have
\begin{equation}
\sup_{x\in \mathbb{R}^d,R>0} \Big\{ \mu_0 ^{\varepsilon} ( B_R (x)) ,\frac{\mu_0 ^{\varepsilon} ( B_R (x)) }{\omega _{d-1}R^{d-1}} \Big\}\leq D_1.
\label{bound}
\end{equation}
\item $\lim _{\varepsilon \to 0}\mu _0 ^{\varepsilon} = \mathcal{H}^{d-1} \lfloor M_0$ as Radon measures, that is  $\lim _{\varepsilon \to 0}\int _{\mathbb{R}^d} \phi \, d\mu _0 ^{\varepsilon} = \int _{M_0} \phi \, d\mathcal{H}^{d-1}$ for any $\phi \in C_c (\mathbb{R}^d)$.
\item $ \lim _{\varepsilon \to 0}\varphi _0 ^{\varepsilon } =2\chi _{U_0} -1$ in $BV _{loc}$.
\item For any $\varepsilon >0$, we have
\begin{equation}
\frac{\varepsilon |\nabla \varphi ^{\varepsilon}_0| ^2}{2}\leq \frac{W(\varphi ^{\varepsilon }_0)}{\varepsilon } \quad \text{on} \quad \Omega.
\label{negativity}
\end{equation}
\end{enumerate}
\end{proposition}
Set
\[
  v^\varepsilon := \begin{cases}
    \frac{-\varphi^\varepsilon _t}{|\nabla\varphi^\varepsilon|} \frac{\nabla\varphi^\varepsilon}{|\nabla\varphi^\varepsilon|} & \text{if} \ |\nabla\varphi^\varepsilon|\not=0 ,\\
    0 & \text{otherwise}.
  \end{cases}
\]
The main result of this paper is the following:
\begin{theorem}\label{mainresults}
Let $d=2,3$ and $U_0\subset \Omega$ be a open set, where $U_0$ satisfies \eqref{initialdata1} and \eqref{initialdata2}. Assume that $\varphi^\varepsilon$ is a solution for \eqref{ac} and the initial data satisfies \eqref{initial}. Then there exists a subsequence $\varepsilon \to 0$ such that the following hold:
\begin{enumerate}
\item[(a)] There exists a family of $(d-1)$-integral Radon measures $\{ \mu _t \}_{t\in [0,\infty)}$ on $\mathbb{R}^d$ such that
\begin{enumerate}
\item[(a1)] $\mu ^\varepsilon \to \mu$ as Radon measures on $\mathbb{R}^d \times [0,\infty)$, where $d\mu :=d\mu_t dt$.
\item[(a2)] $\mu ^\varepsilon _t \to \mu _t$ as Radon measures on $\mathbb{R}^d $ for any $t \in [0,\infty)$.
\end{enumerate}
\item[(b)] There exists $\psi \in BV _{loc}( \Omega \times [0,\infty)) \cap C^{\frac{1}{2}} _{loc} ([0,\infty) ;L^1 (\Omega))$ such that
\begin{enumerate}
\item[(b1)] $\varphi ^{\varepsilon} \to 2\psi -1 \ \ \text{in} \ L^1 _{loc} ( \Omega \times [0,\infty))$ and a.e. pointwise.
\item[(b2)] $\psi(\cdot,0) =\chi _{U_0} $ a.e. on $\Omega$.
\item[(b3)] (Volume preserving property 1) $\psi (\cdot,t)$ is a characteristic function with 
\[ \int_{\Omega} \psi (\cdot,t) \, dx=\mathcal{L}^d (U_0)\]
for any $t\in [0,\infty)$.
\item[(b4)] $\| \nabla \psi (\cdot,t) \| (\phi) \leq \mu _t (\phi) $ for any $t\in [0,\infty)$ and $\phi \in C_c (\mathbb{R}^d;\mathbb{R}^+)$. Moreover $\spt\| \nabla \psi (\cdot,t) \| \subset \spt \mu _t  $ for any $t\in [0,\infty)$.
\end{enumerate}
\item[(c)] There exist $\epsilon \in (0,1)$ and $\lambda \in L^2 _{loc}(0,\infty)$ such that for any $T>0$, we have
\begin{equation}
\sup _{\varepsilon \in (0,\epsilon)} \int _0 ^T |\lambda^\varepsilon (t)|^2 \, dt<\infty \quad \text{and} \quad \lambda^\varepsilon \to \lambda \ \ \text{weakly in} \ L^2(0,T). 
\label{thmc}
\end{equation}
\item[(d)] There exists $g \in L^2 (\mu;\mathbb{R}^d)$ such that 
\begin{equation}
\begin{split}
&\lim _{\varepsilon \to 0} \frac{1}{\sigma}\int _{\mathbb{R}^d \times (0,\infty)}  -\lambda ^\varepsilon \sqrt{2W(\varphi^\varepsilon)}\nabla\varphi^\varepsilon \cdot \Phi \, dxdt = \int _{\mathbb{R}^d \times (0,\infty)} g \cdot \Phi \, d\mu
\end{split}
\label{limitlambda}
\end{equation}
for any $\Phi \in C_c (\mathbb{R}^d \times[0,\infty);\mathbb{R}^d)$.
\item[(e)] $\{ \mu _t \}_{t\in (0,\infty)}$ is a $L^2$-flow with a generalized velocity vector 
\begin{equation*}
v=h+g
\end{equation*}
and
\begin{equation*}
\lim _{\varepsilon \to 0} \int _{\mathbb{R}^d \times (0,\infty)}  v^\varepsilon  \cdot \Phi \, d\mu ^\varepsilon = \int _{\mathbb{R}^d \times (0,\infty)} v \cdot \Phi \, d\mu
\end{equation*}
for any $\Phi \in C_c (\mathbb{R}^d \times[0,\infty);\mathbb{R}^d)$. Moreover there exists a measurable function $\theta : \partial ^{\ast} \{ \psi =1 \}\to \mathbb{N}$ such that 
\begin{equation}
v=h-\frac{1}{\theta} \lambda \nu \qquad \mathcal{H}^{d}\text{-a.e. on} \ \partial ^{\ast} \{\psi =1 \},
\label{generalizedvelocity2}
\end{equation}
where $\nu$ is the inner unit normal vector of $\{ \psi(\cdot,t) =1 \}$ on $\partial ^{\ast} \{ \psi(\cdot,t) =1 \}$.
\item[(f)] (Volume preserving property 2)
\begin{equation*}
\int _\Omega v \cdot \nu \, d\| \nabla\psi (\cdot ,t) \|=0
\end{equation*}
for a.e. $t \in (0,\infty)$.
\end{enumerate}
\end{theorem}
\section{Energy estimates}
In this section, we show the $L^2$ estimates for the multiplier $\lambda ^\varepsilon$. We also prove the monotonicity formula and the density upper bounds for \eqref{ac} by using the negativity of the discrepancy measure $\xi ^\varepsilon _t$. In this section we assume that $d\geq2$, $U_0\subset \Omega$ is a open set, where $U_0$ satisfies \eqref{initialdata1} and \eqref{initialdata2} and $\varphi^\varepsilon$ is a solution for \eqref{ac} and the initial data satisfies \eqref{initial}. Note that we do not need the assumption $d=2,3$ except Remark \ref{remmc} in this section.

\bigskip

We define a Radon measure $\xi _t ^{\varepsilon}$ by
\begin{equation*}
\xi _t ^{\varepsilon}(\phi) := \frac{1}{\sigma}\int _{\mathbb{R}^d} \phi \Big( \frac{\varepsilon |\nabla \varphi ^{\varepsilon}|^2}{2} - \frac{W (\varphi ^{\varepsilon} )}{\varepsilon} \Big) dx,
\end{equation*}
for any $\phi \in C_c (\mathbb{R}^d)$. $\xi _t ^{\varepsilon}$ is called the discrepancy measure. Set $\xi _{\varepsilon} =\xi _{\varepsilon}(x,t):=\frac{\varepsilon |\nabla \varphi ^{\varepsilon}(x,t)|^2}{2} - \frac{W (\varphi ^{\varepsilon}(x,t) )}{\varepsilon}$.
\begin{proposition}\label{negative}
$\xi _{\varepsilon} (x,t)\leq 0$ for any $(x,t) \in \mathbb{R}^d \times [0,\infty)$. Moreover $\xi _t ^{\varepsilon} $ is a non-positive measure for $t \in [0,\infty)$.
\end{proposition}
\begin{proof}
We define a function $r:\mathbb{R}^d \times [0,\infty) \to \mathbb{R}$ by
\[ \varphi ^{\varepsilon} (x,t) =q^{\varepsilon} (r(x,t)) .\]
By \eqref{q} we have
\[ \frac{\varepsilon |\nabla \varphi ^{\varepsilon} |^2/2}{W(\varphi ^{\varepsilon}) / \varepsilon}\leq |\nabla r|^2 \ \ \text{on} \ \ \mathbb{R}^d \times[0,\infty).\]
Hence, if  $|\nabla r|\leq 1$ then $\displaystyle \frac{\varepsilon |\nabla \varphi ^{\varepsilon} |^2 /2}{W(\varphi ^{\varepsilon}) / \varepsilon}\leq 1$.
Thus we only need to prove that $|\nabla r|\leq 1$ on $\mathbb{R}^d \times [0,\infty)$. 

Let $g (q):=k'(q)=\sqrt{2W (q)}$ for $q \in \mathbb{R}$. By \eqref{q} we have
\begin{equation}
q^{\varepsilon} _r =\frac{ g (q^{\varepsilon})}{\varepsilon} \qquad \text{and} \qquad q^{\varepsilon} _{rr}= \frac{( g (q^{\varepsilon}) )_r}{\varepsilon}=\frac{g _q ( q^{\varepsilon} ) }{\varepsilon} q^{\varepsilon} _r.
\label{q2} 
\end{equation}
By \eqref{ac} and \eqref{q2} we obtain
\begin{equation*}
\begin{split}
 q_r ^{\varepsilon} r_t &= q_r ^{\varepsilon} \Delta r + q_{rr} ^{\varepsilon} |\nabla r|^2 -q_{rr} ^{\varepsilon} +\lambda ^\varepsilon q^\varepsilon _r \\
&=  q_r ^{\varepsilon} \Delta r + q_{r} ^{\varepsilon} \frac{g _q}{\varepsilon} (|\nabla r|^2 -1) +\lambda ^\varepsilon q^\varepsilon _r .
\end{split}
\end{equation*}
Thus we have
\[ r_t =  \Delta r + \frac{g _q}{\varepsilon} (|\nabla r|^2 -1) +\lambda ^\varepsilon \]
and
\begin{equation}
\partial _t |\nabla r|^2=\Delta |\nabla r |^2 - 2|\nabla^2 r|^2 +\frac{2}{\varepsilon} \nabla r\cdot \nabla g _q (|\nabla r|^2 -1) + \frac{2g _q}{\varepsilon} \nabla r\cdot \nabla |\nabla r|^2. 
\label{max}
\end{equation}
Note that $\nabla \lambda^\varepsilon =0$. By the assumption we have $|\nabla r(\cdot,0)|=|\nabla \overline{r_\varepsilon}|\leq 1$ on $\mathbb{R}^d$. By \eqref{max} and the maximal principle we obtain $|\nabla r|\leq 1$ in $\mathbb{R}^d \times [0,\infty)$.
\end{proof}
By \eqref{deri2} and \eqref{bound} we have
\begin{proposition}
For any $0\leq t_1 <t_2<\infty$ and $\varepsilon >0$ we have
\begin{equation}
\mu _{t_2} ^\varepsilon (\Omega) +\frac{1}{\sigma}\int_{t_1} ^{t_2} \int _{\Omega} \varepsilon |\partial _t \varphi ^\varepsilon |^2 \, dxdt = \mu ^\varepsilon _{t_1} (\Omega) \leq D_1.
\label{bound2}
\end{equation}
\end{proposition}
By an argument similar to that in \cite{bronsard-stoth}, we obtain the following key estimate:
\begin{proposition}\label{keyprop}
There exist $c_1=c_1 (\omega,D_1) >0$ and $\epsilon_1=\epsilon_1(\omega) >0$ such that
\begin{equation}
\sup _{\varepsilon \in (0,\epsilon_1)} \int _0 ^T |\lambda^\varepsilon (t)|^2 \, dt \leq c_1 (1+T).
\label{L2est}
\end{equation}
\end{proposition}
\begin{proof}
Let $\Omega _\delta \subset \mathbb{R}^d$ be a open set with a smooth boundary $\partial \Omega_\delta$ such that $\Omega \subset \Omega_\delta$ and $\lim_{\delta\to 0}\dist (\partial \Omega_\delta ,\Omega)=0$. Let $\zeta =(\zeta ^1,\zeta^2,\dots ,\zeta ^d)\in C ^1 ( \Omega_\delta ;\mathbb{R}^d )$ satisfy $\zeta \cdot n =0$ on $\partial \Omega_\delta$, where $n$ is the outer unit normal vector of $\Omega_\delta$. Multiply \eqref{ac} by $ \nabla \varphi ^\varepsilon \cdot \zeta$ and integrate over $\Omega_\delta$. Then by the integration by parts we have
\begin{equation}
\begin{split}
& \int _{\Omega_\delta} \varepsilon\varphi ^\varepsilon _t \nabla \varphi ^\varepsilon \cdot \zeta \, dx +\sum _{i,j=1} ^d \int _{\Omega_\delta} \varepsilon \varphi ^\varepsilon _{x_i} \varphi ^\varepsilon _{x_j} \zeta ^j _{x_i} \, dx -\int _{\Omega_\delta } \Big( \frac{\varepsilon |\nabla \varphi ^{\varepsilon}|^2}{2} + \frac{W (\varphi ^{\varepsilon} )}{\varepsilon} \Big)  \div \zeta \, dx\\
=&-\lambda ^\varepsilon \int _{\Omega_\delta} k(\varphi ^\varepsilon) \div \zeta \, dx.
\end{split}
\label{321}
\end{equation}
Note that since $\nabla ( k(\varphi^\varepsilon) ) =\sqrt{2W(\varphi ^\varepsilon)} \nabla \varphi ^\varepsilon $ we have
\[ \int_{\Omega_\delta} \lambda^\varepsilon \sqrt{2W(\varphi^\varepsilon)} \nabla \varphi ^\varepsilon \cdot \zeta  \, dx = \lambda^\varepsilon \int _{\Omega_\delta} \nabla (k(\varphi ^\varepsilon)) \cdot \zeta \, dx = -\lambda^\varepsilon \int _{\Omega_\delta} k(\varphi ^\varepsilon) \div \zeta \, dx.\]

Let $\Psi _\delta \in C_c ^\infty (B_\delta (0))$ be a Dirac sequence and $\eta$ be a solution of the following equation:
\begin{equation}
\left\{ 
\begin{array}{ll}
-\Delta \eta =k(\varphi ^\varepsilon (\cdot,t) ) \ast \Psi _\delta - \dashint _{\Omega } ( k(\varphi ^\varepsilon (\cdot ,t) ) \ast \Psi _\delta)  & \text{in} \ \Omega_\delta, \\
\frac{\partial \eta}{\partial n} =0 & \text{on} \ \partial \Omega_\delta, \\
\dashint _{\Omega } \eta \, dx=0 &
\end{array} \right.
\label{322}
\end{equation}
for $t\geq 0$. By the standard arguments of the elliptic PDE and \eqref{sup1} we have
\begin{equation*}
\|\eta(\cdot ,t) \|_{C^{2,\alpha} (\Omega_\delta)} \leq c_2 , \quad t\geq 0,
\end{equation*}
where $c_2=c_2 (\delta,d) >0$. Set $\zeta := \nabla \eta$. Then by \eqref{321}, \eqref{322} and \eqref{323} we have
\begin{equation}
\begin{split}
&|\lambda^\varepsilon| \Big| \int _{\Omega_\delta} k(\varphi ^\varepsilon) \Big(  k(\varphi ^\varepsilon) \ast \Psi _\delta - \dashint _{\Omega_\delta } k(\varphi ^\varepsilon ) \ast \Psi _\delta \Big) \, dx \Big| 
= |\lambda^\varepsilon| \Big| \int _{\Omega_\delta} k(\varphi ^\varepsilon) (-\div \zeta ) \, dx \Big| \\
\leq & \Big( \int  _{\Omega_\delta } \varepsilon |\nabla \varphi ^\varepsilon |^2 \, dx \Big) ^{\frac{1}{2}} 
\Big( \int  _{\Omega_\delta} \varepsilon (\varphi _t ^\varepsilon )^2 \, dx \Big) ^{\frac{1}{2}} \| \zeta \|_{C^0 (\Omega_\delta) } +2\| \zeta \|_{C^1 (\Omega_\delta )  }  \mu _t ^\varepsilon (\Omega_\delta )\\
\leq &c_3 \Big( \Big( \int  _{\Omega_\delta } \varepsilon (\varphi _t ^\varepsilon )^2 \, dx\Big)^{\frac{1}{2}} +1 \Big),
\end{split}
\label{323}
\end{equation}
where $ c_3=c_3(\delta,D_1)>0$. We compute that 
\begin{equation}
k^2(s) -\frac{4}{9} \geq -W(s)  
\label{324}
\end{equation}
for $|s|\leq 1$.
By using the arguments of \cite[p161]{stoth} and $\|\varphi^\varepsilon\|_{\infty}\leq 1$, there exists $C >0$ such that
\begin{equation}
\begin{split}
&\sup _{|a|\leq\delta} \int _{\Omega_\delta} |k(\varphi^\varepsilon (x+a,t) )-k(\varphi ^\varepsilon (x,t))| \, dx\\
\leq&2\sup _{|a|\leq\delta} \int _{\Omega_\delta} |\varphi^\varepsilon (x+a,t) -\varphi ^\varepsilon (x,t)| \, dx 
\leq C \delta^{\frac{1}{2}}.
\end{split}
\label{327}
\end{equation}
By \eqref{omega2}, \eqref{324} and \eqref{327} we compute that
\begin{equation}
\begin{split}
& \int _{\Omega_\delta } k(\varphi ^\varepsilon) \Big(  k(\varphi ^\varepsilon) \ast \Psi _\delta - \dashint _{\Omega_\delta } k(\varphi ^\varepsilon ) \ast \Psi _\delta \Big) \, dx \\
=& \frac{4}{9} \mathcal{L}^d(\Omega_\delta)  + \int _{\Omega_\delta } k^2 (\varphi ^\varepsilon ) -\frac{4}{9} \, dx -\int _{\Omega_\delta } (k(\varphi^\varepsilon) -k(\varphi^\varepsilon) \ast \Psi _\delta) k(\varphi^\varepsilon) \, dx\\
&-\frac{1}{\mathcal{L}^d(\Omega_\delta) }  \Big( \int _{\Omega_\delta } k (\varphi ^\varepsilon ) \, dx \Big)^2 +  \frac{1}{\mathcal{L}^d(\Omega_\delta)} \int _{\Omega_\delta } k (\varphi ^\varepsilon ) \, dx \int _{\Omega _\delta} k(\varphi ^\varepsilon)-k(\varphi ^\varepsilon) \ast \Psi _\delta \, dx\\
\geq &  \frac{1}{\mathcal{L}^d(\Omega_\delta) } \Big( \frac{4}{9} \mathcal{L}^d(\Omega_\delta) ^2 - \Big( \int _{\Omega_\delta} k (\varphi ^\varepsilon ) \, dx \Big)^2 \Big) - \int _{\Omega_\delta} W(\varphi ^\varepsilon) \, dx \\
&-C\sup _{|a|\leq\delta} \int _{\Omega_\delta} |k(\varphi^\varepsilon (x+a,t)) -k(\varphi ^\varepsilon (x,t))| \, dx\\
\geq & \frac{1}{\mathcal{L}^d(\Omega_\delta) } \Big( \frac{4}{9} \mathcal{L}^d(\Omega_\delta) ^2 - \Big( \int _{\Omega_\delta} k (\varphi ^\varepsilon ) \, dx \Big)^2 \Big) -C(\varepsilon +\delta ^{\frac{1}{2}} ) \geq \frac{\omega}{3},
\end{split}
\label{326}
\end{equation}
for sufficiently small $\varepsilon$ and $\delta$. By \eqref{323} and \eqref{326} we obtain
\begin{equation*}
|\lambda^\varepsilon| \leq 2 \omega^{-1} c_3 \Big( \Big( \int  _{\Omega } \varepsilon (\varphi _t ^\varepsilon )^2 \, dx\Big)^{\frac{1}{2}} +1 \Big)
\end{equation*}
and
\begin{equation*}
\begin{split}
\int _{0} ^T |\lambda^\varepsilon|^2\, dt &\leq 4 \omega^{-2} c ^2 _3 \int _ 0 ^T \Big( \int  _{\Omega } \varepsilon (\varphi _t ^\varepsilon )^2 \, dx+1 \Big) \, dt \\
&\leq 4 \omega^{-2} c ^2 _3 (\sigma \mu _{0} ^\varepsilon (\Omega) +T)\leq 4 \omega^{-2} c ^2 _3 (\sigma D_1 +T),
\end{split}
\end{equation*}
where \eqref{bound2} is used.
\end{proof}
\begin{proposition}\label{proph}
There exist $c_4=c_4(\omega,D_1)>0$ such that
\begin{equation}
\int _0 ^T \int _{\Omega} \varepsilon \Big( \Delta \varphi ^{\varepsilon} -\frac{W' (\varphi ^{\varepsilon})}{\varepsilon ^2} \Big)^2 \, dx dt \leq c_4(1+T)
\label{bound3}
\end{equation}
for any $T>0$ and $\varepsilon \in (0,\epsilon_1)$.
\end{proposition}
\begin{proof}
By \eqref{ac}, \eqref{bound2} and \eqref{L2est} we have
\begin{equation*}
\begin{split}
&\int _0 ^T \int _{\Omega} \varepsilon \Big( \Delta \varphi ^{\varepsilon} -\frac{W' (\varphi ^{\varepsilon})}{\varepsilon ^2} \Big)^2 \, dx dt\\ 
\leq &\int _0 ^T \int _{\Omega} \varepsilon (\varphi_t ^\varepsilon)^2 \, dx dt + \int _0 ^T \int _{\Omega} \varepsilon \Big( \lambda^\varepsilon \frac{\sqrt{2W(\varphi^\varepsilon)}}{\varepsilon}\Big)^2 \, dx dt\\
\leq &D_1 +\int _0 ^T (\lambda^\varepsilon)^2 \int _{\Omega} \frac{2W(\varphi^\varepsilon)}{\varepsilon} \, dx dt
\leq D_1 + 2D_1\int _0 ^T (\lambda^\varepsilon)^2 dt \\
\leq& D_1(1+2c_1(1+T)),
\end{split}
\end{equation*}
where $\int _{\Omega} \frac{2W(\varphi^\varepsilon)}{\varepsilon} \, dx \leq 2\mu_t^\varepsilon(\Omega) \leq 2 D_1$ is used.
\end{proof}
\begin{remark}\label{remmc}
By \eqref{bound3}, we have
\begin{equation}
\liminf_{\varepsilon \to 0} \int _{\Omega} \varepsilon \Big( \Delta \varphi ^{\varepsilon} -\frac{W' (\varphi ^{\varepsilon})}{\varepsilon ^2} \Big)^2 \, dx <\infty \quad \text{for a.e.} \ t\geq 0.
\label{boundh2}
\end{equation}
Assume that $d=2,3$ and there exists a family of Radon measures $\{ \mu_t \}_{t \in [0,\infty)}$ such that $\mu _t ^\varepsilon \to \mu _t$ for $t\geq 0$. Then by \eqref{bound2}, \eqref{boundh2} and \cite[Theorem 4.1]{roger-schatzle}, we have
\[ \int _{\Omega} |h|^2 \, d\mu _t \leq \liminf_{\varepsilon \to 0} \int _{\Omega} \varepsilon \Big( \Delta \varphi ^{\varepsilon} -\frac{W' (\varphi ^{\varepsilon})}{\varepsilon ^2} \Big)^2 \, dx <\infty \quad \text{for a.e.} \ t\geq 0,\]
where $h$ is the generalized mean curvature for $\mu _t$.
\end{remark}
\begin{remark}\label{difficulty}
For \eqref{rs}, \cite{bronsard-stoth} proved the boundedness of $\sup _{\varepsilon} \int _0 ^T (\lambda _1 ^\varepsilon)^2 \, dt $. But  we need the boundedness of $\sup _{\varepsilon} \varepsilon ^{-1}\int _0 ^T (\lambda _1 ^\varepsilon)^2 \, dt $ to obtain \eqref{bound3} for \eqref{rs}.
\end{remark}
Next we show the monotonicity formula. Define
\[ \rho= \rho _{y,s} (x,t) := \frac{1}{(4\pi (s-t))^{\frac{d-1}{2}}} e^{-\frac{|x-y|^2}{4(s-t)}}, \qquad t<s, \ x,y\in \mathbb{R}^d. \]
To localize the monotonicity formula, we choose a radially symmetric cut-off function $\eta(x) \in C_c ^\infty (B_{\frac{1}{2}} (0))$ with $\eta=1$ on $B_{\frac{1}{4}}(0)$ and $0\leq \eta \leq 1$. Define
\[ \tilde \rho _{(y,s)}(x,t):=\rho _{(y,s)}(x,t) \eta (x-y)=\frac{1}{(4\pi (s-t))^{\frac{d-1}{2}}} e^{-\frac{|x-y|^2}{4(s-t)}}\eta (x-y), \qquad t<s, \ x,y\in \mathbb{R}^d.  \]

\begin{proposition}\label{propmono}
There exists $c_5>0$ depending only on $d$ such that
\begin{equation}
\begin{split}
\int _{\mathbb{R}^d} \tilde \rho \, d\mu ^\varepsilon _t (x)\Big|_{t=t_2} \leq & \Big( \int _{\mathbb{R}^d} \tilde \rho \, d\mu ^\varepsilon _t (x) \Big|_{t=t_1} + c_5 \int _{t_1} ^{t_2}e^{-\frac{1}{128(s-t)}} \mu _t ^\varepsilon (B_{\frac{1}{2}}(y)) \, dt\Big)e^{\int_{t_1} ^{t_2} |\lambda^\varepsilon (t)|^2 \, dt}
\end{split}
\label{monoton}
\end{equation}
for $y\in \mathbb{R}^d$ and $0\leq t_1 <t_2$.
\end{proposition}
\begin{proof}
First we show
\begin{equation}
\begin{split}
\frac{d}{dt} \int _{\mathbb{R}^d}\tilde \rho \, d\mu_t ^\varepsilon\leq \frac{1}{2(s-t)}\int _{\mathbb{R}^d}\tilde \rho \, d\xi _t ^\varepsilon
+\frac{1}{2} (\lambda^\varepsilon) ^2 \int _{\mathbb{R}^d} \tilde\rho \, d\mu_t ^\varepsilon +c_5 e^{-\frac{1}{128(s-t)}} \mu _t ^\varepsilon (B_{\frac{1}{2}} (y)).
\end{split}
\label{mono5}
\end{equation}
Define
\[ L^\varepsilon:=\varphi ^\varepsilon _t -\lambda^\varepsilon \frac{\sqrt{2W(\varphi^\varepsilon)}}{\varepsilon} = \Delta \varphi^\varepsilon -\frac{W'(\varphi ^\varepsilon)}{\varepsilon ^2} .\] 
Let $e_\varepsilon :=\frac{\varepsilon |\nabla \varphi ^\varepsilon|^2}{2}+ \frac{W(\varphi ^\varepsilon)}{\varepsilon}$. By integration by parts we obtain
\begin{equation}
\begin{split}
&\frac{d}{dt} \int _{\mathbb{R}^d} e_{\varepsilon}\tilde \rho \, dx 
= \int _{\mathbb{R}^d}\{ e_\varepsilon \tilde\rho _t -\varepsilon \Big(L^\varepsilon+\lambda^\varepsilon \frac{\sqrt{2W(\varphi^\varepsilon)}}{\varepsilon}\Big)(\nabla \tilde\rho \cdot \nabla \varphi^\varepsilon +\tilde\rho L^\varepsilon) \} \, dx \\
=& \int _{\mathbb{R}^d} \Big\{ e_{\varepsilon} \tilde\rho _t -\varepsilon\tilde \rho \Big(L^\varepsilon+\frac{\nabla\tilde\rho \cdot \nabla \varphi^\varepsilon  }{\tilde\rho } \Big)^2 +\varepsilon \Big( L^\varepsilon\nabla \tilde\rho \cdot \nabla \varphi ^\varepsilon+\frac{(\nabla\tilde \rho \cdot \nabla \varphi^\varepsilon)^2}{\rho} \Big)\\
&\qquad -\varepsilon \tilde\rho \lambda \frac{\sqrt{2W(\varphi^\varepsilon)}}{\varepsilon} \Big(L^\varepsilon+\frac{\nabla\tilde\rho \cdot \nabla \varphi ^\varepsilon }{\tilde\rho } \Big) \Big\} \, dx \\
\leq & \int _{\mathbb{R}^d} \Big\{ e_\varepsilon \tilde\rho _t +\varepsilon \Big( L^\varepsilon\nabla\tilde \rho \cdot \nabla \varphi ^\varepsilon+\frac{(\nabla \tilde\rho \cdot \nabla \varphi^\varepsilon)^2}{\rho} \Big)
+\frac{1}{4} \varepsilon \tilde\rho \Big(\lambda ^\varepsilon\frac{\sqrt{2W(\varphi^\varepsilon)}}{\varepsilon} \Big)^2 \Big\} \, dx .
\end{split}
\label{mono1}
\end{equation}
Furthermore by integration by parts we have
\begin{equation}
\begin{split}
\int_{\mathbb{R}^d} \varepsilon L^\varepsilon \nabla \tilde\rho \cdot\nabla \varphi^\varepsilon \, dx
= \int  _{\mathbb{R}^d} -\varepsilon (\nabla \varphi^\varepsilon\otimes \nabla \varphi^\varepsilon) : \nabla ^2\tilde \rho + e_{\varepsilon} \Delta \tilde\rho \, dx.
\end{split}
\label{mono2}
\end{equation}
Substitution of \eqref{mono1} into \eqref{mono2} gives
\begin{equation}
\begin{split}
\frac{d}{dt} \int _{\mathbb{R}^d} e_{\varepsilon}\tilde \rho \, dx \leq & \int _{\mathbb{R}^d}(-\xi _\varepsilon)(\partial_t \tilde \rho+\Delta \tilde \rho) +\varepsilon |\nabla \varphi^\varepsilon|^2 \Big( \partial_t \tilde \rho+\Delta \tilde \rho - \frac{\nabla \varphi^\varepsilon\otimes \nabla \varphi^\varepsilon}{|\nabla\varphi^\varepsilon|^2} : \nabla ^2\tilde \rho\\ &+\frac{(\nabla\tilde\rho\cdot \nabla\varphi^\varepsilon)^2}{\tilde\rho|\nabla\varphi^\varepsilon|^2} \Big)\,dx
+\frac{1}{2} (\lambda^\varepsilon) ^2 \int _{\mathbb{R}^d} e_\varepsilon \tilde\rho \, dx .
\end{split}
\label{mono4}
\end{equation}

Note that $\rho $ (without multiplication $\eta$) satisfies the following:
\begin{equation}
\begin{split}
\rho _t +\Delta \rho = -\frac{\rho}{2(s-t)} ,\qquad \rho _t +\Delta \rho -\frac{\nabla \varphi ^\varepsilon\otimes \nabla \varphi ^\varepsilon}{|\nabla \varphi^\varepsilon |^2} :\nabla ^2 \rho + \frac{(\nabla \rho \cdot \nabla \varphi^\varepsilon )^2}{\rho |\nabla \varphi ^\varepsilon|^2} =0.
\end{split}
\label{mono3}
\end{equation}
When one computes \eqref{mono3} with $\tilde \rho $ instead of $\rho$, we obtain additional term coming from differentiation of $\eta$. Note that $|\nabla ^j \rho|\leq C(j,d) e^{-\frac{1}{128(s-t)}} $ for any $x,y\in \mathbb{R}^d$ with $|x-y|>\frac{1}{4}$ and $j=0,1$. Hence the integration of these terms can be estimated by $C\mu _t ^\varepsilon (B_{\frac{1}{2}} (y)) e^{-\frac{1}{128(s-t)}}$ for $C=C(d)>0$. Thus we obtain \eqref{mono5}. By \eqref{mono5}, Proposition \ref{negative} and Gronwall's inequality we have \eqref{monoton}.
\end{proof}

Next we prove the upper density ratio bounds of $\mu _t ^{\varepsilon}$.
\begin{proposition}\label{propdens}
There exists $c_6=c_6(d)>0$ such that
\begin{equation}
\frac{\mu _t ^{\varepsilon} (B_R(x))}{ R^{d-1} }\leq c_6 D_1 ( 1 +  t )e^{c_1 (1+t)}
\label{density}
\end{equation}
for $(x,t) \in \Omega \times [0,\infty)$, $\varepsilon \in (0,\epsilon_1)$ and $0<R<\frac{1}{4}$.
\end{proposition}
\begin{proof}
We compute that
\begin{equation}
\begin{split}
&\int _{\mathbb{R}^d}\tilde\rho _{y,s}(x,0) \, d\mu _0 ^{\varepsilon}(x) \leq \frac{1}{(4\pi s)^{\frac{d-1}{2}} } \int _{\mathbb{R}^d} e^{-\frac{|x-y|^2}{4s}} \, d\mu _0 ^{\varepsilon} (x)\\
=& \frac{1}{(4\pi s)^{\frac{d-1}{2}} } \int _0 ^1 \mu _0 ^{\varepsilon}( \{ x \, | \, e^{-\frac{|x-y|^2}{4s}}>k \}) \, dk 
= \frac{1}{(4\pi s)^{\frac{d-1}{2}} } \int _0 ^1 \mu _0 ^{\varepsilon}( B_{\sqrt{4s\log k^{-1}}} (y)) \, dk\\
\leq &  \frac{1}{(4\pi s)^{\frac{d-1}{2}} } \int _0 ^1 D_1 \omega_{d-1} (\sqrt{4s\log k^{-1}})^{d-1}  \, dk \leq c_7 D_1,
\end{split}
\label{density0}
\end{equation}
where $c_7 >0$ is depending only on $d$ and the density upper bound \eqref{bound} is used. By \eqref{bound2}, \eqref{monoton} and \eqref{density0}, we have
\begin{equation}
\begin{split}
&\int _{\mathbb{R}^d} \tilde \rho_{y,s} (x,t) \, d\mu ^\varepsilon _t (x) \\
\leq& \Big( \int _{\mathbb{R}^d} \tilde \rho_{y,s}(x,0) \, d\mu ^\varepsilon _0 (x) + c_5 \int _{0} ^{t}e^{-\frac{1}{128(s- \tau )}} \mu _\tau ^\varepsilon (B_{\frac{1}{2}}(y)) \, d\tau \Big)e^{\int_{0} ^{t} |\lambda^\varepsilon (\tau)|^2 \, d\tau}\\
\leq& \Big( c_7 D_1 + c_5 D_1 t \Big)e^{c_1 (1+t)}
\end{split}
\label{density1}
\end{equation}
for any $0<t<s$ and $y\in\Omega$. Fix $R\in(0,\frac{1}{4})$ and set $s=t+\frac{R^2}{4}$. Then
\begin{equation}
\begin{split}
&\int _{\mathbb{R}^d}\tilde \rho _{y,s}(x,t) \, d\mu _t ^{\varepsilon}(x)=\int _{\mathbb{R}^d} \frac{1}{\pi ^{\frac{d-1}{2}} R^{d-1}} e^{-\frac{|x-y|^2}{R^2}} \eta(x-y) \, d \mu _t ^{\varepsilon}(x)\\
\geq &\int _{B_R (y)} \frac{1}{\pi ^{\frac{d-1}{2}} R^{d-1}} e^{-\frac{|x-y|^2}{R^2}} \, d \mu _t ^{\varepsilon}
\geq  \int _{B_R (y)} \frac{1}{\pi ^{\frac{d-1}{2}} R^{d-1}} e^{-1} \, d \mu _t ^{\varepsilon}=\frac{1}{e \pi ^{\frac{-1}{2}} R^{d-1}}\mu _t^{\varepsilon} (B_R (y)),
\end{split}
\label{density2}
\end{equation}
where $\eta (x-y)=1$ on $B_R(y)$ is used. By \eqref{density1} and \eqref{density2} we obtain \eqref{density}.
\end{proof}
\section{Existence of $L^2$-flow}
To prove Theorem \ref{mainresults}, we use the following theorem:
\begin{theorem}[\cite{mugnai-roger2011}]\label{mr}
Let $d=2,3$ and $\varphi^\varepsilon$ be a solution for the following equation:
\begin{equation}
\left\{ 
\begin{array}{ll}
\varepsilon \varphi ^{\varepsilon} _t =\varepsilon \Delta \varphi ^{\varepsilon} -\dfrac{W' (\varphi^{\varepsilon})}{\varepsilon }+g^\varepsilon ,& (x,t)\in \Omega \times (0,\infty).  \\
\varphi ^{\varepsilon} (x,0) = \varphi_0 ^{\varepsilon} (x) ,  &x\in \Omega.
\end{array} \right.
\label{acmr}
\end{equation}
We assume that there exists $\epsilon >0$ such that
\begin{equation*}
\sup_{\varepsilon \in (0,\epsilon)} \Big( \mu _0 ^\varepsilon (\Omega) + \int _{\Omega \times (0,T)} \frac{1}{\varepsilon} (g^\varepsilon)^2 \, dxdt\Big) <\infty
\end{equation*}
for any $T>0$. Then there exits a subsequence $\varepsilon \to0$ such that the following hold:
\begin{enumerate}
\item There exists a family of $(d-1)$-integral Radon measures $\{ \mu _t \}_{t\in [0,\infty)}$ on $\mathbb{R}^d$ such that
\begin{enumerate}
\item $\mu ^\varepsilon \to \mu$ as Radon measures on $\mathbb{R}^d \times [0,\infty)$, where $d\mu =d\mu_t dt$.
\item $\mu ^\varepsilon _t \to \mu _t$ as Radon measures on $\mathbb{R}^d$ for any $t \in [0,\infty)$.
\end{enumerate}
\item There exists $g \in L^2 (\mu;\mathbb{R}^d)$ such that
\begin{equation}
\begin{split}
\lim _{\varepsilon \to 0} \frac{1}{\sigma}\int _{\mathbb{R}^d \times (0,\infty)} -g^\varepsilon\nabla \varphi^\varepsilon \cdot \Phi  \, dxdt = \int _{\mathbb{R}^d \times (0,\infty)} g \cdot \Phi  \, d\mu
\end{split}
\label{mr2}
\end{equation}
for any $\Phi \in C_c (\mathbb{R}^d \times[0,\infty);\mathbb{R}^d)$.
\item $\{ \mu _t \}_{t\in (0,\infty)}$ is a $L^2$-flow with a generalized velocity vector $v=h+g$ and
\begin{equation*}
\lim _{\varepsilon \to 0} \int _{\mathbb{R}^d \times (0,\infty)}  v^\varepsilon  \cdot \Phi \, d\mu ^\varepsilon = \int _{\mathbb{R}^d \times (0,\infty)} v \cdot \Phi \, d\mu  
\end{equation*}
for any $\Phi \in C_c (\mathbb{R}^d \times[0,\infty);\mathbb{R}^d)$, where
\[
  v^\varepsilon := \begin{cases}
    \frac{-\varphi^\varepsilon _t}{|\nabla\varphi^\varepsilon|} \frac{\nabla\varphi^\varepsilon}{|\nabla\varphi^\varepsilon|} & \text{if} \ |\nabla\varphi^\varepsilon|\not=0 ,\\
    0 & \text{otherwise}.
  \end{cases}
\]
\end{enumerate}
\end{theorem}
\begin{remark}
The boundary conditions of \eqref{acmr} of the original theorem is Neumann conditions. But by an argument similar to the proof, we also obtain same results for periodic boundary conditions (see \cite[Remark 2.3]{mugnai-roger2008}). 
\end{remark}

\begin{proof}[Proof of Theorem \ref{mainresults}]
Set $g^\varepsilon := \lambda ^\varepsilon \sqrt{2W(\varphi ^\varepsilon)}$. By \eqref{bound2} and \eqref{L2est} we have \eqref{thmc} and
\begin{equation}
\begin{split}
\sup _{\varepsilon \in (0,\epsilon_1)} \int _0 ^T \int _{\Omega} \frac{1}{\varepsilon} (g^\varepsilon )^2 \, dxdt
= \sup _{\varepsilon \in (0,\epsilon_1)} \int _0 ^T  (\lambda ^\varepsilon )^2 \int _{\Omega} \frac{2W(\varphi ^\varepsilon)}{\varepsilon} \, dxdt 
\leq  2c_1D_1(1+T).
\end{split}
\label{L2est2}
\end{equation}
By \eqref{bound} and \eqref{L2est2}, $\varphi ^\varepsilon$ satisfy all the assumptions of Theorem \ref{mr}. Then (a) of Theorem \ref{mainresults} holds and $\{ \mu_t \}_{t\in[0,\infty)}$ is a $L^2$-flow with a generalized velocity vector $v=h+g$ with \eqref{limitlambda}. As a supplement, we show the key estimate \eqref{velo2} directly in Proposition \ref{prvelo2}. So we only need to prove (b), \eqref{generalizedvelocity2} and (f).

Next we prove (b). Set $w^i := G \circ \varphi ^i$, where $G  (s) := \sigma ^ {-1} \int _{-1 } ^s \sqrt{2W (y) } \, dy $ and $\varphi ^i := \varphi^{\varepsilon _i}$. Note that $G (-1) =0$ and $G (1) =1$. We compute that
\[ |\nabla w^i |=\sigma ^{-1} |\nabla \varphi ^i| \sqrt{2W (\varphi ^i) } \leq \sigma ^{-1} \Big( \frac{\varepsilon_i |\nabla \varphi ^i |^2 }{2} +\frac{W (\varphi ^i)}{\varepsilon_i} \Big). \]
Hence by \eqref{bound2} we have
\begin{equation}
\int _{\Omega} |\nabla w^i (\cdot,t) | \, dx \leq \int _{\Omega} \sigma ^{-1} \Big( \frac{\varepsilon_i |\nabla \varphi ^i |^2 }{2} +\frac{W (\varphi ^i)}{\varepsilon_i} \Big) \, dx = \mu _t ^{\varepsilon_i} (\Omega) \leq D_1
\label{bv1}
\end{equation}
for $t\geq 0$. Fix $T>0$. By the similar argument and \eqref{bound2} we obtain
\begin{equation}
\begin{split}
\int _0 ^T \int_{\Omega} |\partial _t w^i| \, dxdt \leq \sigma ^{-1} \int _0 ^T \int_{\Omega} \Big( \frac{\varepsilon_i |\partial _t \varphi ^i |^2 }{2} +\frac{W (\varphi ^i)}{\varepsilon_i} \Big) \, dxdt \leq D_1(1+T).
\end{split}
\label{bv2}
\end{equation}
By \eqref{bv1} and \eqref{bv2}, $\{ w^i \}_{i=1} ^\infty$ is bounded in $BV  (\Omega \times [0,T])$. By the standard compactness theorem and the diagonal argument there is subsequence $\{ w^i \}_{i=1} ^\infty$ (denoted by the same index) and $w\in BV_{loc} (\Omega \times [0,\infty))$ such that
\begin{equation}
w^i \to w \quad \text{in} \ L^1 _{loc} (\Omega \times [0,\infty))
\label{wconv}
\end{equation}
and a.e. pointwise. We denote $\psi(x,t) := \lim _{i\to \infty} (1+G^{-1} \circ w^i (x,t))/2$. Then we have
\[ \varphi ^i \to 2\psi -1 \quad \text{in} \ L^1 _{loc} (\Omega \times [0,\infty)) \]
and a.e. pointwise. Hence we obtain (b1). By Proposition \ref{prop1} (3) we obtain (b2). We have $\varphi ^i \to \pm 1$ a.e. and $\psi=1$ or $=0$ a.e. on $\Omega \times [0,\infty)$ by the boundedness of $\int _{\Omega} \frac{W (\varphi ^i) }{\varepsilon_i} \, dx$. Moreover $\psi =w$ a.e. on $\Omega \times [0,\infty)$. Thus $\psi \in BV_{loc} (\Omega \times [0,\infty))$. For any open set $U \subset \Omega$ and a.e. $0\leq t_1 < t_2<T$ we have
\begin{equation}
\begin{split}
&\int_U |\psi (\cdot ,t_2) -\psi (\cdot ,t_1)| \, dx \leq\lim _{i\to \infty} \int _{\Omega} |w^i (\cdot ,t_2) -w^i (\cdot ,t_1)| \, dx \\
\leq & \liminf _{i \to \infty} \int_\Omega \int _{t_1} ^{t_2} |\partial _t w^i| \, dtdx \leq \liminf_{i\to \infty} \int_{\Omega} \int _{t_1} ^{t_2} \Big( \frac{\varepsilon_i |\partial _t \varphi ^i |^2 }{2}\sqrt{t_2 -t_1} +\frac{W (\varphi ^i)}{\varepsilon_i \sqrt{t_2 -t_1}} \Big) \, dtdx \\
\leq & C_2 D_1 \sqrt{t_2 -t_1},
\end{split}
\label{holder}
\end{equation}
where $C_2= C_2 (n,T)>0$. By \eqref{holder} and $\mathcal{L}^d (U_0)<\infty$, $\psi (\cdot ,t) \in L^1 (\mathbb{R}^d)$ for a.e. $t\geq 0$. By this and \eqref{holder}, we may define $\psi (\cdot ,t)$ for any $t\geq 0$ such that $\psi\in C^{\frac{1}{2}} _{loc} ([0,\infty) ; L^1 (\Omega))$. Moreover by (b2), \eqref{h} and $\psi=1 $ or $=0$ a.e. on $\Omega\times [0,\infty)$ we have $\int _{\Omega} \psi (\cdot,t) \, dx =\mathcal{L}^d (U_0)$ for any $t\geq 0$. Hence we obtain (b3). For $\phi \in C_c (\Omega ;\mathbb{R}^+)$ and $t\geq 0$ we compute that
\begin{equation*}
\begin{split}
&\int _{\Omega} \phi \, d\| \nabla \psi(\cdot , t) \| \leq \liminf _{i \to \infty} \int _{\Omega} \phi |\nabla w^i | \, dx\\
\leq & \lim _{i\to \infty}\sigma ^{-1} \int _{\Omega} \phi \Big( \frac{\varepsilon_{i} |\nabla \varphi ^i |^2 }{2} +\frac{W (\varphi ^i)}{\varepsilon_{i}} \Big) \, dx = \int _{\Omega} \phi \, d\mu _t.
\end{split}
\end{equation*}
Hence we obtain (b4).

Next we prove \eqref{generalizedvelocity2}. By \eqref{mr2}, for any $\Phi \in C_c ^1 (\mathbb{R}^d \times[0,\infty);\mathbb{R}^d)$ we have
\begin{equation}
\begin{split}
&\int _{\Omega \times (0,\infty)} g \cdot \Phi  \, d\mu=\lim _{i\to \infty} \frac{1}{\sigma}\int _{\Omega \times (0,\infty)} -\lambda^i \sqrt{2W(\varphi^i)}\nabla \varphi^i \cdot \Phi  \, dxdt \\
=& \lim _{i\to \infty} \frac{1}{\sigma}\int _{\Omega \times (0,\infty)} -\lambda^i \nabla k(\varphi^i) \cdot \Phi  \, dxdt
=\lim _{i\to \infty} \frac{1}{\sigma}\int _{\Omega \times (0,\infty)}\lambda^i k(\varphi^i)  \div \Phi \, dxdt.
\end{split}
\label{rn}
\end{equation}
Set $\varphi:=\lim_{i \to \infty}\varphi ^i$. By \eqref{rn}, the Radon-Nikodym theorem and 
\[ \lim_{i\to \infty} k(\varphi ^i)=\lim_{i\to \infty} \int_0 ^{\varphi^i} \sqrt{2W(s)} \, ds=\frac{\sigma}{2} \varphi =\sigma\Big(\psi-\frac{1}{2}\Big) \quad\text{ a.e. on } \ \Omega \times (0,\infty),\]
we obtain
\begin{equation}
\begin{split}
&\int _{\Omega \times (0,\infty)} g \cdot \Phi  \, d\mu
=\int_{ 0}^\infty \lambda \int _{\Omega} \Big(\psi-\frac{1}{2}\Big)\div \Phi \, dxdt
=\int_{ 0}^\infty \lambda \int _{\Omega} \psi\div \Phi \, dxdt\\
&=-\int_{ 0}^\infty \lambda \int _{\Omega} \nu\cdot  \Phi  \, d\| \nabla \psi(\cdot,t)\|dt
=\int _{\Omega\times (0,\infty)} - \lambda \frac{d\|\nabla\psi(\cdot,t)\|}{d\mu_t} \nu  \cdot \Phi \, d\mu
\end{split}
\label{limitg}
\end{equation}
for any $\Phi \in C_c ^1 (\mathbb{R}^d \times[0,\infty);\mathbb{R}^d)$, where $\nu$ is the inner normal vector of $\{ \psi(\cdot,t)=1 \}$ on $\partial ^\ast \{ \psi(\cdot,t)=1\}$. Set $\theta : \partial ^\ast \{ \psi = 1 \}\to (0,\infty)$ by $\theta := \Big( \frac{d\|\nabla\psi(\cdot,t)\|}{d\mu_t} \Big)^{-1}$. By the integrality of $\mu_t$, $\theta\in\mathbb{N}$ $\mathcal{H}^d$-a.e. Hence we have \eqref{generalizedvelocity2}.

Finally we prove (f). By \cite[Proposition 4.5]{mugnai-roger2008} we have
\begin{equation}
\int _0 ^T \int_{\mathbb{R}^d} v \cdot \nu \eta \, d\|\nabla \varphi (\cdot, t) \| dt=-\int _0 ^T \int_{\mathbb{R}^d} \varphi \eta _t \, dxdt
\label{peri}
\end{equation}
for any $T>0$ and $\eta \in C_c ^1 (\mathbb{R}^d \times (0,T))$. By \eqref{peri}, (b3) and the periodic boundary condition, 
\begin{equation}
\int _0 ^T \zeta \int_{\Omega} v \cdot \nu  \, d\|\nabla \varphi (\cdot, t) \| dt=-\int _0 ^T \zeta _t \int_{\Omega} \varphi  \, dxdt = -(2\mathcal{L}^d (U_0) -1) \int _0 ^T \zeta _t \, dt=0 
\label{peri2}
\end{equation}
holds for any $\zeta \in C_c ^1 ((0,T))$. 
By \eqref{peri2} and $\|\nabla \psi (\cdot,t)\|=\frac{1}{2} \|\nabla \varphi (\cdot,t)\|$ for any $t\geq0$, we have
\begin{equation*} 
\begin{split}
&\int _{0} ^{T} \zeta \int _{\Omega} v \cdot \nu \, d\|\nabla \psi (\cdot,t)\| dt
=\frac{1}{2}\int _{0} ^{T} \zeta \int _{\Omega} v \cdot \nu \, d\|\nabla \varphi (\cdot,t)\|dt=0
\end{split}
\end{equation*}
for any $\zeta \in C_c ^1 ((0,T))$. Hence we obtain (f).
\end{proof}

\begin{proposition}\label{prvelo2}
Let all the assumptions of Theorem \ref{mainresults} hold and a family of Radon measures $\{\mu _t\} _{t\in [0,\infty)}$ satisfy $\mu ^\varepsilon _t \to \mu _t$ as Radon measures for any $t \in [0,\infty)$. Then there exists a subsequence $\varepsilon \to 0$ such that \eqref{velo2} holds. 
\end{proposition}
\begin{proof}
For any $\eta \in C_c ^1 (\mathbb{R}^d\times (0,T))$ we compute that
\begin{equation}
\begin{split}
&\frac{d}{dt}\int _{\mathbb{R}^d} \eta \, d\mu _t ^\varepsilon
=\int _{\mathbb{R}^d} \eta _t \, d\mu _t ^\varepsilon +\frac{1}{\sigma}\int  _{\mathbb{R}^d} \eta \Big( \varepsilon \nabla \varphi^\varepsilon \cdot \nabla \varphi _t ^\varepsilon +\frac{W' (\varphi^\varepsilon )}{\varepsilon}\varphi _t ^\varepsilon \Big) \, dx \\
=& \int _{\mathbb{R}^d} \eta _t \, d\mu _t ^\varepsilon + \frac{1}{\sigma}\int_{\mathbb{R}^d }  \varepsilon \eta \Big( -\Delta \varphi^\varepsilon +\frac{W' (\varphi ^\varepsilon )}{\varepsilon ^2} \Big)\varphi^\varepsilon _t \, dx
-\frac{1}{\sigma}\int_{\mathbb{R}^d }  \varepsilon (\nabla \eta \cdot \nabla \varphi^\varepsilon )\varphi^\varepsilon _t \, dx\\
=& \int _{\mathbb{R}^d} \eta _t \, d\mu _t ^\varepsilon + \frac{1}{\sigma}\int_{\mathbb{R}^d }  \varepsilon \eta \Big( -\Delta \varphi^\varepsilon +\frac{W' (\varphi ^\varepsilon )}{\varepsilon ^2} \Big)\varphi^\varepsilon _t \, dx
+\int_{\mathbb{R}^d } \nabla \eta \cdot v^\varepsilon \, d\tilde{\mu}_t ^\varepsilon,
\end{split}
\label{deri3}
\end{equation}
where $d\tilde{\mu}_t ^\varepsilon:=\frac{\varepsilon}{\sigma} |\nabla\varphi^\varepsilon|^2 dx$. By \eqref{bound2}, \eqref{bound3} and \eqref{deri3}, there exists $C>0$ such that
\begin{equation}
\begin{split}
& \Big|\int _0 ^T \Big(\int_{\mathbb{R}^d} \eta_t \, d\mu _t ^\varepsilon + \int_{\mathbb{R}^d } \nabla \eta \cdot v^\varepsilon \, d\tilde{\mu}_t ^\varepsilon \Big)dt\Big|\\
\leq & \|\eta\|_{C^0(\mathbb{R}^d\times(0,T))} \Big\{ \mu_0 (\spt \eta) +\mu _T (\spt \eta)\\
& + \frac{1}{\sigma} \Big( \int _0 ^T \int _{\spt \eta} \varepsilon \Big( \Delta \varphi ^{\varepsilon} -\frac{W' (\varphi ^{\varepsilon})}{\varepsilon ^2} \Big)^2 \, dx dt \Big)^\frac{1}{2} \Big( \int _0 ^T \int _{\spt \eta} \varepsilon (\varphi^\varepsilon )^2 \, dx dt \Big)^\frac{1}{2} \Big\}\\
\leq & C\|\eta\|_{C^0(\mathbb{R}^d\times(0,T))}.
\end{split}
\label{deri4}
\end{equation}
Note that by \eqref{bound2}, \eqref{bound3} and \cite[Proposition 4.9]{roger-schatzle} there exists a subsequence $\varepsilon \to 0$ such that $\xi ^\varepsilon _t \to 0$ as Radon measures for a.e. $t \in [0,\infty)$. Hence $\tilde{\mu} ^\varepsilon_t \to \mu_t $ as Radon measures for a.e. $t \in [0,\infty)$. By \eqref{bound2}, we have
\[ \sup _{\varepsilon \in (0,\epsilon_1 )} \int _0 ^T \int_{\mathbb{R}^d } | v^\varepsilon|^2 \, d\tilde{\mu}_t ^\varepsilon dt =\sup _{\in (0,\epsilon_1 )} \int _0 ^T \int_{\mathbb{R}^d } \varepsilon | \varphi ^\varepsilon _t|^2 \, dx dt \leq \sigma D_1. \]
Hence there exist $v\in L^2 (\mu;\mathbb{R}^d)$ and a subsequence $\varepsilon \to 0$ such that
\begin{equation}
\lim _{\varepsilon \to 0} \int _0 ^T \int_{\mathbb{R}^d } \Phi \cdot v^\varepsilon \, d\tilde{\mu}_t ^\varepsilon dt
= \int _0 ^T \int_{\mathbb{R}^d } \Phi \cdot v \, d\mu
\label{pair}
\end{equation}
for any $\Phi \in C_c (\mathbb{R}^d \times (0,\infty) ;\mathbb{R}^d)$ (see \cite[Theorem 4.4.2]{hutchinson}). Hence by \eqref{deri4} and  \eqref{pair} we obtain \eqref{velo2}.
\end{proof}

\end{document}